\documentclass{tbimc}

\newtheorem{ozn}{Definition}
\newtheorem{thm}[ozn]{Theorem}

\begin{document}
\selectlanguage{english}

\title[Consistent estimation in Cox proportional hazards model]
{Consistent estimation in Cox proportional hazards model with measurement errors\\ and unbounded parameter set}

\author{Alexander Kukush}
\address{Taras Shevchenko National University of Kyiv, Volodymyrska st. 64, 01601, Kyiv, Ukraine.} 
\email{alexander\_kukush@univ.kiev.ua}
\author{Oksana Chernova}
\address{Taras Shevchenko National University of Kyiv, Volodymyrska st. 64, 01601, Kyiv, Ukraine.}
\email{chernovaoksan@gmail.com} 

\date{01/03/2017}

\begin{abstract}
Cox proportional hazards model with measurement error is investigated. In Kukush et al. (2011) [Journal of Statistical Research {\bf 45}, 77--94] and Chimisov and Kukush (2014) [Modern Stochastics: Theory and Applications {\bf 1}, 13--32]
asympto\-tic properties of simultaneous estimator $\lambda_n(\cdot)$, $\beta_n$ were studied for baseline hazard rate $\lambda(\cdot)$ and regression parameter $\beta$, at that the parameter set $\Theta=\Theta_{\lambda}\times \Theta_{\beta}$ was assumed bounded.
In the present paper, the set $\Theta_{\lambda}$ is unbounded from above and  not separated away from $0$. We construct the estimator in two steps: first we derive a strongly consistent estimator and then modify it to provide its asymptotic normality.
\end{abstract}
\maketitle

\section{Introduction}
Consider the Cox proportional hazards model (Cox, 1972), 
where a lifetime $T$ has the following intensity function
\begin{equation}
\label{eqindm}
\lambda(t|X;\lambda,\beta)=\lambda(t)\exp(\beta^TX), \quad t\geq0.
\end{equation}
A covariate $X$ is a given random vector distributed in $\mathbb{R}^m$, $\beta$ is a parameter belonging to $\Theta_\beta\subset \mathbb{R}^m$, and $\lambda(\cdot) \in \Theta_\lambda\subset C[0,\tau]$ is a baseline hazard function.

We use common random censorship model: instead of  $T$, only a couple $(Y,\Delta)$ is available, where $Y:=\min \{T,C\}$ and  $\Delta:=I_{\{T\leq C\}}$ is the censorship indicator. The censor $C$ is distributed on  $[0,\tau]$. Its survival function $G_C(u)=1-F_C(u)$ is unknown, while we know $\tau$. The conditional pdf of T given X is
\begin{equation*} \label{pdf T}
f_T(t|X, \lambda, \beta)=\lambda(t | X;\lambda, \beta)\exp \left(-\int_0^t\lambda(t | X;\lambda, \beta) d s\right).
\end{equation*}


Throughout this paper an additive error model is considered, i.e., instead of $X$ a surrogate variable
$$W=X+U$$
is observed, where a random error $U$ has known moment generating function\\
 ${M_U(z):=\mathsf{E} e^{z^TU}}$. A couple $(T, X)$, censor $C,$ and measurement error $U$ are stochastically independent.

Consider independent copies of the model $(X_i, T_i, C_i, Y_i, \Delta_i, U_i, W_i),\ i=1,...,n$. Based on triples $(Y_i,\Delta_i,W_i),\; i=1,...,n$, we estimate true parameters $\beta_0$ and $\lambda_0(t)$, $t \in[0,\tau]$. Due to the suggestion of Augustin (2004)
we use the following objective function 
$$Q_n^{cor}(\lambda,\beta):=\frac{1}{n} \sum_{i=1}^{n} q(Y_i,\Delta_i,W_i;\lambda,\beta),$$
where
$$q(Y,\Delta,W;\lambda,\beta):=\Delta(\log\lambda(Y)+\beta^TW)-\frac{\exp(\beta^TW)}{M_U(\beta)}\int_0^Y \lambda(u)du.$$
The corrected estimator is defined as
\begin{equation}
\label{obfun}
(\hat{\lambda}_n,\hat{\beta}_n)=\arg \max_{(\lambda,\beta)\in \Theta}Q_n^{cor}(\lambda,\beta),
\end{equation}
where $\Theta:=\Theta_{\lambda}\times \Theta_{\beta}$. If the parameter sets $\Theta_{\lambda}$ and $\Theta_{\beta}$ are compact, then $\Theta$ is compact as well and the maximum in (\ref{obfun}) is attained.

The issue of estimating $\beta_0$ and cumulative hazard $\Lambda(t)=\int_0^t\lambda_0(s)ds$ has been extensively studied in
the literature in past decades:
in Andersen and Gill (1982) general ideas are presented based on partial likelihood;
model with measurement errors is considered in Gu and Kong(1999), where, based on Corrected Score method, consistent
and asymptotically normal estimators are constructed for regression parameter and cumulative hazard function;
Royston (2011) discusses some problems where the behavior of baseline hazard function $\lambda_0(\cdot)$ itself, rather than cumulative hazard, is needed.

Our model is presented in Augustin (2004), where the baseline hazard function is assumed to belong to a parameter space, while we consider $\lambda_0(\cdot)$ from a compact set of $C[0,\tau]$.

In \cite{KuBar} the consistency of estimator (\ref{obfun}) is proven for a bounded parameter set. In \cite{ChiKu} its asymptotic normality is presented. We remark that in \cite{KuBar} the authors write $\Theta_\lambda$ without a formal requirement that $\lambda(0)$ is bounded, though actually this assumption was used throughout the paper. We prove that this condition and the separation of $\lambda(\cdot)$ away from zero are too restrictive and can be omitted. 

The paper is organized as follows. In Section 2, we define an estimator under unbounded parameter set and prove its consistency. Additionally, we describe a numerical scheme for calculation of the estimator. In Section 3, we modify the estimator constructed in Section 2 to produce the asymptotically normal estimator, and Section 4 concludes.

\section{Consistent estimation on the first stage}

Impose conditions on the parameter sets. 
\begin{enumerate}
  \item[(i)] $K_\lambda\subset C[0,\tau]$ is the following closed convex set of nonnegative functions
$$K_\lambda:=\{~ f:[0,\tau]\to \mathbb{R} |\; f(t)\geq 0,\forall t \in [0,\tau]\; \text{and}$$$$\; |f(t)-f(s)|\leq L|t-s|,\forall t,s\in [0,\tau]~\},$$
where $L>0$ is a fixed constant.
  \item[(ii)] $\Theta_\beta \subset \mathbb{R}^m$ is a compact set.
\end{enumerate}

The following conditions (iii) -- (vi) are borrowed from \cite{KuBar}. 
\begin{enumerate}
\item[(iii)] $\mathsf{E} U=0$ and for some constant $\epsilon >0$,
\begin{equation*}
\mathsf{E} e^{D\| U \|}<\infty , \ \text{where} \ D:=\max_{\beta \in \Theta_\beta}\| \beta \| +\epsilon.
\end{equation*}
\item[(iv)] $\mathsf{E} e^{D \| X \|}< \infty$, where $D>0$ is defined in (iii).
\item[(v)] $\tau $ is the right endpoint of the distribution of $C$, that is\\
$\mathsf{P}(C>\tau)=0$ and for all $\epsilon >0$,
$\mathsf{P}(C>\tau-\epsilon)>0$.
\item[(vi)] The covariance matrix of random vector $X$ is positive definite.
\end{enumerate}

Denote
\begin{equation}\label{K}
    K=K_\lambda\times \Theta_\beta.
\end{equation}

If $\lambda(Y)=0$ then we put 
$$\Delta\cdot\log\lambda(Y)=\left\{ \begin{array}{rcl} 0& \mbox{if}\quad \Delta=0\;\\
-\infty &\mbox{if}\quad \Delta=1.
\end{array}\right.$$

\begin{ozn}\label{def1}
Fix a sequence $\{\varepsilon_n\}$ of positive numbers, with $\varepsilon_n\downarrow 0$, as $n\to \infty$. The corrected estimator $(\hat{\lambda}^{(1)}_n,\hat{\beta}^{(1)}_n)$ of $(\lambda,\beta)$ is a Borel measurable function of observations $(Y_i,\Delta_i,W_i)$, $i=1,...,n$, with values in $K$ and such that
\begin{equation}\label{def}
    Q_n^{cor}(\hat{\lambda}^{(1)}_n,\hat{\beta}^{(1)}_n)\geq \sup_{(\lambda,\beta)\in K}Q_n^{cor}(\lambda,\beta)-\varepsilon_n.
\end{equation}
\end{ozn}

The corrected estimator exists due to Pfanzagl (1969) (it is essential here that the supremum in (\ref{def}) is finite). 
Additionally, assume the following.
\begin{enumerate}
\item[(vii)]  True parameters $(\lambda_0,\beta_0)$ belong to  $K$, which is given in (\ref{K}), and moreover ${\lambda_0(t)>0}$, $t\in[0,\tau]$.
\end{enumerate}

\begin{ozn}
Let $A_n=A_n(\omega)$, $n=1,2,\ldots$, be a sequence of statements depending on an elementary event $\omega\in\Omega$. We say that $A_n$ holds \textit{eventually} if for almost all $\omega$ there exists $n_0=n_0(\omega)$ such that for all $n\geq n_0(\omega),$ $A_n$ holds true.
\end{ozn}

\begin{thm}\label{ther1}
Under conditions (i) to (vii) the estimator $(\hat{\lambda}^{(1)}_n,\hat{\beta}^{(1)}_n)$ is a strongly consistent estimator of the true parameters  $(\lambda_0,\beta_0)$, that is
\begin{equation*}
    \max_{t\in[0,\tau]} |\hat{\lambda}^{(1)}_n(t)-\lambda_0(t)|\to 0\quad \text{and} \quad \hat{\beta}^{(1)}_n \to \beta_0
\end{equation*}
almost surely, as $n\to \infty$.
\end{thm}
\begin{proof}
For $R>0$ denote
\begin{equation*}
    K_{\lambda}^R=K_{\lambda}\cap\bar{B}(0,R) \quad \text{and} \quad K^R=K_{\lambda}^R\times\Theta_\beta,
\end{equation*}
where $\bar{B}(0,R)$ is  closed ball in $C[0,\tau]$ with center in the origin and radius $R$.

1. In the first part of the proof, we show that for large enough nonrandom $R>||\lambda_0||$, it holds \textit{eventually}
\begin{equation}\label{sup}
    \sup_{(\lambda,\beta)\in K^R}Q_n^{cor}(\lambda,\beta)>\sup_{(\lambda,\beta)\in K\setminus K^R}Q_n^{cor}(\lambda,\beta).
\end{equation}

For $\lambda\in K_{\lambda}$ the Lipschitz condition
implies \begin{equation}\label{Lip}
\lambda(0)-L\tau\leq \lambda(t) \leq \lambda(0)+L \tau,
\end{equation}
therefore,
$$q(Y_i,\Delta_i,W_i;\lambda,\beta)\leq \Delta_i\left(\ln(\lambda(0)+L\tau)+\beta^TW_i\right)-\frac{\exp(\beta^TW_i)Y_i}{M_U(\beta)}(\lambda(0)-L\tau).$$

Using the Lipschitz condition 
 for $\lambda\in K_{\lambda},$ one can show that if ${\lambda(t_1)>R}$ for some $t_1\in[0,\tau]$, then ${\lambda(t)>R-L\tau}$ for all $t\in[0,\tau]$. On the other side,  $\lambda(0)>R$ yields $\lambda(t)>R-L\tau$, $t\in[0,\tau]$.
Thus, supremum on the right hand side of (\ref{sup}) can be taken over the set ${\{\lambda\in K_{\lambda}:\lambda(0)> R\}\times \Theta_{\beta}}$.

Denote
$$D_1=\max\limits_{\beta \in \Theta_\beta}\| \beta \|.$$
We have
$$\sup_{(\lambda,\beta)\in K\setminus K^R} Q_n^{cor}(\lambda,\beta)\leq I_1+\sup_{\substack{\lambda\in K_{\lambda}:\\ \lambda(0)>R}}I_2+I_3,$$
with
$$I_1=-(R-L\tau)\frac{1}{n}\sum_{i:\Delta_i=0}\frac{\exp(-D_1||W_i||)Y_i}{\max\limits_{\beta \in \Theta_\beta}M_U(\beta)},$$
$$I_2=\ln(\lambda(0)+L\tau)\frac{1}{n}\sum_{i:\Delta_i=1}\Delta_i -(\lambda(0)+L\tau)\frac{1}{n}\sum_{i:\Delta_i=1}\frac{\exp(-D_1||W_i||)Y_i}{\max\limits_{\beta \in \Theta_\beta}M_U(\beta)},$$
$$I_3=\frac{1}{n}\sum_{i:\Delta_i=1} D_1||W_i||+2L\tau\frac{1}{n}\sum_{i:\Delta_i=1}\frac{\exp(-D_1||W_i||)Y_i}{\max\limits_{\beta \in \Theta_\beta}M_U(\beta)}.$$
By the strong law of large numbers (SLLN),
$$I_1\to -(R-L\tau)\frac{\mathsf{E} [~C\cdot I(\Delta=0) \exp{(-D_1||W||)}~]}{\max\limits_{\beta \in \Theta_\beta} M_U(\beta)}$$
a.s., as $n \to \infty$.
This means that \textit{eventually}
$$I_1\leq -(R-L\tau)D_2,$$
where $D_2>0$.

 Denote $$A_n=\frac{1}{n}\sum_{i=1}^n \Delta_i,~~  B_n=\frac{1}{n}\sum_{i=1}^n\frac{\exp(-D_1||W_i||)Y_i}{\max\limits_{\beta \in \Theta_\beta} M_U(\beta)}1_{\{\Delta_i=1\}}.$$
Since $A_n>0$ and $B_n>0$ \textit{eventually},  
for $\lambda(0)>R$ we get
$$I_2\leq\max_{z>0} (A_n\log z-zB_n)=A_n\left(\log\left(\frac{A_n}{B_n}\right)-1\right).$$
By the SLLN
$$A_n\to \mathsf{P}(\Delta=1)>0,\quad B_n\to \frac{\mathsf{E} [~T\cdot I(\Delta=1)\exp{(-D_1||W||)}~]}{\max\limits_{\beta \in \Theta_\beta}M_U(\beta)}>0$$
a.s., as  $n \to \infty$. Therefore, $I_2$ \textit{eventually} is bounded from above by some positive constant $D_3$.

Further, using the SLLN it can be shown that \textit{eventually} $I_3$ is also bounded from above by some positive constant $D_4$.
Thus,
\begin{equation*}\label{f}
\limsup_{n \to \infty}  \sup_{(\lambda,\beta)\in K\setminus K^R} Q_n^{cor}(\lambda,\beta)\leq -(R-L\tau)D_2+D_3+D_4.
\end{equation*}
Here $D_2$, $D_3,$ and $D_4$ do not depend on $\beta \in \Theta_\beta$.
Tend $R \to +\infty$ and obtain
\begin{equation*}
\limsup_{n \to \infty} \sup_{(\lambda,\beta)\in K\setminus K^R} Q_n^{cor}(\lambda,\beta) \to -\infty,~ \text{as}\ R \to +\infty.
\end{equation*}
This proves that the inequality (\ref{sup}) holds \textit{eventually} for large enough $R$.

Therefore, we may and do replace $K$ for $K^R$ in Definition 1. Thus, we assume that for all $n\geq1$,
\begin{equation}\label{def2}
    Q_n^{cor}(\hat{\lambda}^{(1)}_n,\hat{\beta}^{(1)}_n)\geq \sup_{(\lambda,\beta)\in K^R}Q_n^{cor}(\lambda,\beta)-\varepsilon_n
\end{equation}
and $(\hat{\lambda}^{(1)}_n,\hat{\beta}^{(1)}_n)\in K^R$. Notice that $K^R$ is a compact set in $C[0,\tau]$.

2. Since $R>||\lambda_0||$, we have $(\lambda_0,\beta_0)\in K^R$. Then (\ref{def2}) implies the inequality
\begin{equation}\label{def0}
    Q_n^{cor}(\hat{\lambda}^{(1)}_n,\hat{\beta}^{(1)}_n)\geq Q_n^{cor}(\lambda_0,\beta_0)-\varepsilon_n.
\end{equation}
We fix $\omega\in A\subset\Omega$, with $\mathsf{P}(A)=1$. Further, we will impose additional conditions on $A$. We want to show that at point $\omega$, $(\hat{\lambda}^{(1)}_n,\hat{\beta}^{(1)}_n)\to(\lambda_0,\beta_0)$.
We have
\begin{equation}\label{convq}
Q_n^{cor}(\lambda_0,\beta_0)\to q_{\infty}(\lambda_0,\beta_0):=\mathsf{E} [q(Y,\Delta,W;\lambda_0,\beta_0)].
\end{equation}
 This holds almost surely, so we can assume (\ref{convq}) for the fixed $\omega$. Therefore, the first condition is
 $$Q_n^{cor}(\lambda_0,\beta_0;\omega)\to q_{\infty}(\lambda_0,\beta_0),\quad \omega \in A.$$
 The sequence $\{(\hat{\lambda}^{(1)}_n(\omega),\hat{\beta}^{(1)}_n(\omega)),n\geq1\}$ belongs to the compact set $K^R$. Consider an arbitrary convergent subsequence
\begin{equation}\label{subs}
    (\hat{\lambda}^{(1)}_{n'}(\omega),\hat{\beta}^{(1)}_{n'}(\omega))\to(\lambda_*,\beta_*) \in K^R.
\end{equation}
Then (\ref{def0}), (\ref{convq}) imply that
$$q_{\infty}(\lambda_0,\beta_0)\leq \liminf_{n'\to \infty} Q_{n'}^{cor}(\hat{\lambda}^{(1)}_{n'},\hat{\beta}^{(1)}_{n'})=$$
$$=\liminf_{n'\to \infty} \frac{1}{n'}\sum_{i=1}^{n'}\Delta_i\log\hat{\lambda}^{(1)}_{n'}(Y_i)+$$$$+\lim_{n'\to\infty}\frac{1}{n'}\sum_{i=1}^{n'}\left(\Delta_i\hat{\beta}^{(1)T}_{n'}W_i-
\frac{\exp(\hat{\beta}^{(1)T}_{n'}W_i)}{M_U(\hat{\beta}^{(1)}_{n'})}\int_0^{Y_i} \hat{\lambda}^{(1)}_{n'}(u)du\right).$$

The next assumption on $A$ is as follows: for all $\omega\in A,$ a  sequence of random functions
$$\frac{1}{n}\sum_{i=1}^n\left(\Delta_i\beta^TW_i-\frac{\exp(\beta^TW_i)}{M_U(\beta)}\int_0^{Y_i} \lambda(u)du\right)$$
converges uniformly in $(\lambda,\beta)\in K^R$ to
$$\mathsf{E} \left[\Delta\beta^TW-\frac{\exp(\beta^TW)}{M_U(\beta)}\int_0^Y \lambda(u)du\right]=:q_{\infty}^2(\lambda,\beta).$$
Such condition can be imposed because for any fixed $(\lambda,\beta)\in K^R~$ the latter sequence converges to $q^2_{\infty}$ a.s.,  the sequence is equicontinuous a.s. on the compact set $K^R,$ and the limit function is continuous on $K^R.$ These three statements ensure that the sequence converges to $q^2_{\infty}$ uniformly on $K^R$.

The function $q^2_{\infty}$ is continuous in $(\lambda,\beta)\in K^R$, thus,
$$q_{\infty}(\lambda_0,\beta_0)\leq \liminf_{n'\to \infty} \frac{1}{n'}\sum_{i=1}^{n'}\Delta_i\ln\hat{\lambda}^{(1)}_{n'}(Y_i)+q_{\infty}^2(\lambda_*,\beta_*).$$

For large $n'$, it holds
\begin{equation*}
    \hat{\lambda}^{(1)}_{n'}(t)\leq \lambda_*(t)+\varepsilon, \;t\in[0,\tau],
\end{equation*}
with fixed $\varepsilon>0.$
We demand also that
$$\frac{1}{n}\sum_{i=1}^n \Delta_i\ln \lambda(Y_i)\to \mathsf{E} [\Delta \lambda(Y)]$$
uniformly in  $(\lambda,\beta)\in \left(K^{R+\delta_k}_\lambda\cap\{\lambda:\lambda(t)\geq \delta_k\}\right)\times \Theta_\beta$, for each $k\geq1$ and $\omega\in A$, where $\delta_k\downarrow0$, $\{\delta_k\}$ is a fixed sequence of positive numbers.
Then by the SLLN
$$\liminf_{n'\to \infty} \frac{1}{n'}\sum_{i=1}^{n'}\Delta_i\log\hat{\lambda}^{(1)}_{n'}(Y_i)\leq \liminf_{n'\to \infty} \frac{1}{n'}\sum_{i=1}^{n'}\Delta_i\log(\lambda_*(Y_i)+\varepsilon)=$$$$=\mathsf{E} [\Delta\cdot\log(\lambda_*(Y)+\varepsilon)]=:q_{\infty}^{1,\varepsilon}(\lambda_*).$$
Hence, for each $\varepsilon>0$,
\begin{equation*}
    q_{\infty}(\lambda_0,\beta_0)\leq q_{\infty}^{1,\varepsilon}(\lambda_*)+q_{\infty}^2(\lambda_*,\beta_*).
\end{equation*}
Now, tend $\varepsilon\to0$. We have
$$q_{\infty}^{1,\varepsilon}(\lambda_*)=\mathsf{E} [~\Delta\cdot\log(\lambda_*(Y)+\varepsilon)I(\lambda_*(Y)>\frac{1}{2})~]+
\mathsf{E} [~\Delta\cdot\log(\lambda_*(Y)+\varepsilon)I(\lambda_*(Y)\leq\frac{1}{2})~].$$
The first expectation tends to
$$\mathsf{E} [~\Delta\cdot\log(\lambda_*(Y))I(\lambda_*(Y)>\frac{1}{2})~]$$
 by the Lebesgue dominance convergence theorem, and
the second expectation tends to $$\mathsf{E} [~\Delta\cdot\log(\lambda_*(Y))I(\lambda_*(Y)\leq\frac{1}{2})~]$$
by the Lebesgue monotone convergence theorem. Then
$$q_{\infty}^{1,\varepsilon}(\lambda_*,\beta_*)\to q_{\infty}^{1}(\lambda_*,\beta_*):=\mathsf{E} [~\Delta\cdot\log\lambda_*(Y)~],$$
as $\varepsilon\to0$. Therefore,
$$q_{\infty}(\lambda_0,\beta_0)\leq q_{\infty}^{1}(\lambda_*)+q_{\infty}^{2}(\lambda_*,\beta_*)=q_{\infty}(\lambda_*,\beta_*).$$
But according to \cite{KuBar} the inequality
\begin{equation*}
    q_{\infty}(\lambda_0,\beta_0)\geq q_{\infty}(\lambda_*,\beta_*)
\end{equation*}
holds true, moreover the equality is attained if, and only if, $\lambda_*=\lambda_0$ and $\beta_*=\beta_0$. Therefore, a convergent subsequence (\ref{subs}) converges exactly to $(\lambda_0,\beta_0)$. Since the whole sequence belongs to a compact set, this implies the convergence
$$(\hat{\lambda}^{(1)}_n(\omega),\hat{\beta}^{(1)}_n(\omega))\to (\lambda_0,\beta_0),\quad\text{as}\quad n\to \infty.$$
This holds for almost every $\omega\in \Omega$, and the strong consistency is proven.\\
\end{proof}

Now, we explain how the estimator can be computed. Like in \cite{ChiKu} we prove that for a fixed $\beta\in \Theta_\beta,$ the function $\hat{\lambda}^{(1)}_n$ that maximizes $Q_n^{cor}$ is a linear spline. 
\begin{thm}
Under conditions (i) and (ii) the function $\hat{\lambda}^{(1)}_n,$ which maximizes $Q_n^{cor},$ is a linear spline.
\end{thm}
\begin{proof}
Let $(Y_{i_1},...,Y_{i_n})$ be a variational series of $Y_1,.., Y_n$. Fix $\beta\in \Theta_\beta$. Suppose that we are given $\hat{\lambda}^{(1)}_n\in \Theta_\lambda$ that maximizes $Q_n^{cor}(\cdot,\beta)$. 
Together with $(\hat{\lambda}^{(1)}_n, \beta)$ consider $(\bar{\lambda}_n, \beta)$, where $\bar{\lambda}_n$ is the following function. We set
$\bar{\lambda}_n(Y_{i_k})=\hat{\lambda}^{(1)}_n(Y_{i_k})$, $k=1,...,n$. At each interval $[Y_{i_k},Y_{i_{k+1}}]$, $k=1,...,n-1$,
draw straight lines 
$$L^1_{i_{k}}(t)=\hat{\lambda}^{(1)}_n(Y_{i_k})+L(Y_{i_k}-t),$$
$$L^2_{i_{k}}(t)=\hat{\lambda}^{(1)}_n(Y_{i_{k+1}})+L(t-Y_{i_{k+1}}),$$
with $L$  defined in (i). Denote by $B_{i_k}$ the intersection of $L^1_{i_{k}}(t)$ and $L^2_{i_{k}}(t)$. Additionally, $B_{i_0}:=0$, $B_{i_{n}}:=\tau$, $Y_{i_0}:=0$, $Y_{i_{n+1}}:=\tau$. Finally, define the function $~~\bar{\lambda}_n(t)$ as
\begin{equation}\label{spl}
    \bar{\lambda}_n(t)=\left\{ \begin{array}{rcl}  L^2_{i_0}(t)& \mbox{if} & t \in [0,Y_{i_1}],\\
    \max\{L^1_{i_{k}}(t), 0\}& \mbox{if} & t\in [Y_{i_k}, B_{i_k}], \;k=1,...,n-1,\\
 	\max\{L^2_{i_{k}}(t), 0\}& \mbox{if} & t\in [B_{i_k}, Y_{i_{k+1}}],\; k=1,...,n-1, \\
    L^1_{i_n}(t)& \mbox{if} & t \in [Y_{i_n},\tau].\\
                \end{array}\right.
\end{equation}
It is easily seen that $\bar{\lambda}_n \in \Theta_\lambda$. By construction, $\hat{\lambda}^{(1)}_n \geq \bar{\lambda}_n$. Thus,
\begin{equation*}
    Q_n^{cor}(\hat{\lambda}^{(1)}_n, \beta)\leq Q_n^{cor}(\bar{\lambda}_n, \beta).
\end{equation*}
This yields $\hat{\lambda}^{(1)}_n = \bar{\lambda}_n$ and completes the proof.\\
\end{proof}

Notice that \textit{eventually} $\bar{\lambda}_n(B_{i_k})>0,$ and then one can omit  maximum in (\ref{spl}).

As soon as we constructed a linear spline
$$\bar{\lambda}_n(\beta)=\arg \max_{\lambda:(\lambda,\beta)\in\Theta} Q_n^{cor},$$ we maximize $Q(\beta):=Q_n^{cor}(\bar{\lambda}_n(\beta),\beta)$ in $\beta\in \Theta_\beta$, i.e., we search for a $\hat{\beta}\in \Theta_\beta$ such that
\begin{equation*}
    Q(\hat{\beta})\geq \sup_{\beta\in\Theta_\beta} Q(\beta)-\varepsilon_n.
\end{equation*}
Since $Q(\beta)$ is bounded, $\hat{\beta}$ does exist.

We have
\begin{equation*}\label{}
    Q_n^{cor}(\bar{\lambda}_n(\hat{\beta}),\hat{\beta})\geq \sup_{\beta\in\Theta_\beta}\max_{\lambda\in\Theta_\lambda} Q(\beta)-\varepsilon_n=\sup_{ (\lambda,\beta)\in\Theta} Q_n^{cor}(\lambda,\beta)-\varepsilon_n.
\end{equation*}
Therefore, the estimator $(\bar{\lambda}_n(\hat{\beta}),\hat{\beta})$ satisfies Definition \ref{def1} and its evaluation is just a parametric problem.

\section{Construction of asymptotically normal estimator on the second stage}
In this section, we modify the estimator $(\hat{\lambda}^{(1)}_n(\omega),\hat{\beta}^{(1)}_n(\omega))$ from Definition 1 in order to produce the asymptotically normal estimator.
\begin{ozn} The modified corrected estimator $(\hat{\lambda}^{(2)}_n,\hat{\beta}^{(2)}_n)$ of $(\lambda,\beta)$ is a Borel measurable function of observations $(Y_i,\Delta_i,W_i)$, $i=1,...,n$, with values in $K$ and such that
\begin{equation*}\label{def3}
    (\hat{\lambda}^{(2)}_n,\hat{\beta}^{(2)}_n)=
    \begin{cases} \arg\max\{~Q_n^{cor}(\lambda,\beta)|\ (\lambda,\beta)\in K,\;\mu_\lambda\geq\frac{1}{2}\mu_{\hat{\lambda}^{(1)}_n}~\},
    &\mbox{if~}\; \mu_{\hat{\lambda}^{(1)}_n}>0,\\
    (\hat{\lambda}^{(1)}_n,\hat{\beta}^{(1)}_n),&  \ \mbox{otherwise}, 
    \end{cases}
\end{equation*}
with $\mu_\lambda:=\min_{t\in[0,\tau]} \lambda(t)$.
\end{ozn}

 Such estimator exists due to results of  Pfanzagl \cite{Pfanzagl}. Notice that by Theorem \ref{ther1} ${\mu_{\hat{\lambda}^{(1)}_n}\to \mu_{\lambda_0}>0}$ a.s., and \textit{eventually} it holds
$$K_1:=\{(\lambda,\beta)\in K|\mu_\lambda\geq\frac{3}{4}\mu_{\lambda_0}\}\subset\{(\lambda,\beta)\in K|\mu_\lambda\geq\frac{1}{2}\mu_{\hat{\lambda}^{(1)}_n}\} \subset$$
$$\subset\{(\lambda,\beta)\in K|\mu_\lambda\geq\frac{1}{4}\mu_{\lambda_0}\}=:K_2.$$
The estimator
\begin{equation}
    (\hat{\lambda}^{(3)}_n,\hat{\beta}^{(3)}_n)= \arg\max_{(\lambda,\beta)\in K_2} Q_n^{cor}(\lambda,\beta)
\end{equation}
is strongly consistent under the conditions (i) -- (vii), because according to Theorem \ref{ther1} \textit{eventually} it can be taken as an estimator $(\hat{\lambda}^{(1)}_n,\hat{\beta}^{(1)}_n)$. Therefore ${(\hat{\lambda}^{(3)}_n,\hat{\beta}^{(3)}_n)\in K_1}$ \textit{eventually}, and
\textit{eventually} $(\hat{\lambda}^{(3)}_n,\hat{\beta}^{(3)}_n)$ can be taken as an estimator $(\hat{\lambda}^{(2)}_n,\hat{\beta}^{(2)}_n)$. This implies the strong consistency of $(\hat{\lambda}^{(2)}_n,\hat{\beta}^{(2)}_n)$.

Introduce additional assumptions, under which the estimator $(\hat{\lambda}^{(2)}_n,\hat{\beta}^{(2)}_n)$ is asympto\-tically normal.
\begin{enumerate}
\item[(viii)] $\beta_0$ is an interior point of $\Theta_\beta$.
\item[(ix)] $\lambda_0\in \Theta_\lambda^\epsilon$ for some $\epsilon>0$, where\\
$$\Theta_\lambda^\epsilon:=\{~ f:[0,\tau]\to \mathbb{R} |\; f(t)\geq \epsilon,\ \forall t \in [0,\tau],$$
$$\ |f(t)-f(s)|\leq (L-\epsilon)|t-s|,\forall t,s\in [0,\tau] ~\}.$$
\item[(x)] $\mathsf{P}(C>0) =1$.
\item[(xi)] $ \mathsf{E} U=0$ and for some $\epsilon >0,$
\begin{equation*}
 \mathsf{E} e^{2D\| U \|}<\infty ,  \ D:=\max_{\beta \in \Theta_\beta}\| \beta \| +\epsilon.
\end{equation*}
\item[(xii)] $ \mathsf{E} e^{2D \| X \|}< \infty$, with $D$  defined in (xi).
\end{enumerate}
Further, we use  notations from \cite{ChiKu}. Let
$$a(t)=\mathsf{E}[Xe^{\beta_0^T X}G_T(t|X)], \quad b(t)=\mathsf{E}[e^{\beta_0^T X}G_T(t|X)],$$
$$p(t)=\mathsf{E}[XX^Te^{\beta_0^T X}G_T(t|X)], \quad
T(t)=p(t)b(t)-a(t)a^T(t), \quad K(t)=\frac{\lambda_0(t)}{b(t)},$$
$$A=\mathsf{E}\left[XX^T e^{\beta_0^T X} \int_0^Y \lambda_0 (u) du\right], \quad M=\int_0^{\tau} T(u) K(u)G_c(u)du.$$

For $i=1,2,\ldots,$ introduce random variables
$$\zeta_i=-\frac{\Delta_i a(Y_i)}{b(Y_i)}+\frac{\exp(\beta_0^TW_i)}{M_U(\beta_0)}\int_0^{Y_i}a(u)K(u)du+\frac{\partial q}{\partial \beta}(Y_i,\Delta_i,W_i,\beta_0,\lambda_0),$$
with
$$\frac{\partial q}{\partial \beta}(Y,\Delta,W;\lambda,\beta)=\Delta \cdot W-\frac{M_U(\beta)W-E(U e^{\beta^TU})}{M_U(\beta)^2}
\exp(\beta^TW)\int_0^Y \lambda(u)du.$$
Let $$\Sigma_{\beta}=4\cdot \mathsf{Cov}(\zeta_1),~~  m(\varphi_\lambda)=\int_0^{\tau}\varphi_\lambda(u)a(u)G_C(u)du,$$
$$\sigma^2_{\varphi}=4\cdot \mathsf{Var}~ \langle q'(Y,\Delta,W,\lambda_0,\beta_0),\varphi\rangle=$$
$$=4\cdot Var\left[ \frac{\Delta \cdot \varphi_\lambda(Y)}{\lambda_0(Y)}-\frac{\exp(\beta_0^T W)}{M_U(\beta_0)}\int_0^Y \varphi_\lambda(u)du
+\Delta \cdot \varphi_{\beta}^T  W+\right .$$
$$\left.+ \varphi_{\beta}^T \frac{M_U(\beta_0)W- E(U e^{\beta^T_0U})}{M_U(\beta_0)^2}
\exp(\beta_0^T W)\int_0^Y \lambda_0(u)du\right]
$$
with $\varphi=(\varphi_{\lambda},\varphi_{\beta})\in C[0,\tau]\times R^m$,
where $q'$ denotes the Fr\'echet derivative.

Now, we can apply Theorem 1 from \cite{ChiKu} to state the asymptotic normality for $\hat{\beta}^{(2)}_n$ and $\hat{\lambda}^{(2)}_n$
(it follows from the asymptotic normality of consistent estimators $\hat{\beta}^{(3)}_n$ and $\hat{\lambda}^{(3)}_n$).

\begin{thm} Assume conditions (i), (ii), (v) -- (xii). Then $M$ is nonsingular and
$$\sqrt{n}(\hat{\beta}^{(2)}_n-\beta_0)\xrightarrow{\text{d}}N_m(0,M^{-1}\sigma_{\beta}M^{-1}).$$
Moreover, for any Lipschitz continuous function $f$ on $[0,\tau]$,
\begin{equation*}
\sqrt{n}\int_0^\tau(\hat{\lambda}^{(2)}_n-\lambda_0)(u)f(u)G_C(u) du\xrightarrow{\text{d}}N(0,\sigma^2_{\varphi}(f)),
\end{equation*}
where $\sigma^2_{\varphi}(f)=\sigma^2_{\varphi}$ with $\varphi=(\varphi_{\lambda},\varphi_{\beta})$, $\varphi_{\beta}=-A^{-1}m(\varphi_{\lambda})$, and $\varphi_\lambda$ is a unique solution to the Fredholm's integral equation
\begin{equation*}
\frac{\varphi_\lambda}{K(u)}-a^T(u)A^{-1}m(\varphi_{\lambda})=f(u).
\end{equation*}
\end{thm}

 For computation of estimator $(\hat{\lambda}^{(2)}_n,\hat{\beta}^{(2)}_n)$ we refer to \cite{ChiKu}.

\section{Conclusion}
 Under quite mild assumptions, we construct  the estimator for the function $\lambda(\cdot)$ and parameter $\beta$ in Cox proportional hazards model with measurement errors. Contrary to Kukush et al. (2011) and Chimisov and Kukush (2014), we consider an unbounded parameter set. The obtained estimator is consistent and can be modified to be asymptotically normal. Also, we  describe a numerical scheme for calculation of the estimator.
In future we intend to construct confidence regions based on the estimator.


\end{document}